\documentclass[reqno,12pt]{amsart}
\usepackage{amsmath,amsthm,amssymb,amsfonts,amscd}
\usepackage{epsfig}
\usepackage{xypic}
\numberwithin{equation}{section}

\theoremstyle{plain}
\newtheorem{theorem}{Theorem}

\newtheorem{lemma}[theorem]{Lemma}
\newtheorem{corollary}[theorem]{Corollary}
\newtheorem{proposition}[theorem]{Proposition}
\newtheorem*{theorem*}{Theorem}
\newtheorem*{conjecture*}{Conjecture}

\theoremstyle{definition}
\newtheorem{remark}[theorem]{Remark}

\newtheorem*{definition}{Definition}

\newcommand{\CC}{{\mathbb{C}}}

\newcommand{\QQ}{{\mathbb{Q}}}

\newcommand{\RR}{{\mathbb{R}}}
\newcommand{\ZZ}{{\mathbb{Z}}}

\def\C{{\mathcal C}}

\def\H{{\mathcal H}}

\def\O{{\mathcal O}}

\begin{document}
\title[Variance of the exponents]{Variance of the exponents of orbifold Landau-Ginzburg models}
\author{Wolfgang Ebeling and Atsushi Takahashi}
\address{Institut f\"ur Algebraische Geometrie, Leibniz Universit\"at Hannover, Postfach 6009, D-30060 Hannover, Germany}
\email{ebeling@math.uni-hannover.de}
\address{
Department of Mathematics, Graduate School of Science, Osaka University, 
Toyonaka Osaka, 560-0043, Japan}
\email{takahashi@math.sci.osaka-u.ac.jp}
\date{}
\begin{abstract} We prove a formula for the variance of the set of exponents of a non-degenerate weighted homogeneous polynomial with an action of a diagonal subgroup of ${\rm SL}_n(\CC)$.
\end{abstract}
\maketitle
\section*{Introduction}

Let $X$ be a smooth compact K\"ahler manifold of dimension $n$. 
The Hodge numbers $h^{p,q}(X):=\dim_\CC H^q(X,\Omega_X^p)$, $p,q \in \ZZ$, are some  of the most important numerical invariants 
of $X$. 
They satisfy
\[
h^{p,q}(X)=h^{q,p}(X),\quad p,q\in\ZZ,
\]
and the Serre duality
\[
h^{p,q}(X)=h^{n-p,n-q}(X),\quad p,q\in\ZZ.
\]
The Euler number $\chi(X)$ can also be written in terms of the Hodge numbers as
\[
\chi(X)=\sum_{p,q\in\ZZ}(-1)^{p+q} h^{p,q}(X).
\]
One can easily calculate the expectation value of the distribution $\{q \in \ZZ \, | \, h^{p,q}(X) \neq 0\}$, which is given by the formula
\[
\sum_{p,q\in\ZZ}(-1)^{p+q} q\cdot h^{p,q}(X)=\frac{1}{2}n\cdot\chi(X).
\]
Equivalently, this can be rewritten as 
\[
\sum_{p,q\in\ZZ}(-1)^{p+q} \left(q- \frac{n}{2} \right) h^{p,q}(X)=0.
\]
This means nothing else but that the mean of the distribution $\{q \in \ZZ \, | \, h^{p,q}(X) \neq 0\}$ is $n/2$.
It is then natural to ask what is the variance of this distribution. A formula for this variance was given by A.~Libgober and J.~Wood \cite{LW} and L.~Borisov \cite{B}: 
\begin{theorem}[Libgober-Wood, Borisov] \label{thm:lwb}
One has
\begin{equation}
\sum_{p,q\in\ZZ}(-1)^{p+q} \left(q- \frac{n}{2} \right)^2 h^{p,q}(X)
=\frac{1}{12}n\cdot \chi(X)+\frac{1}{6}\int_X c_1(X)\cup c_{n-1}(X),
\end{equation}
where $c_i(X)$ denotes the $i$-th Chern class of $X$. 
\end{theorem}
If the first Chern class $c_1(X)$ is numerically zero, then the above formula becomes
\begin{equation}
\sum_{p,q\in\ZZ}(-1)^{p+q} \left(q- \frac{n}{2} \right)^2 h^{p,q}(X)
=\frac{1}{12}n\cdot \chi(X).
\end{equation}
Similar phenomena were discovered in singularity theory. 
Let us consider a polynomial $f(x_1,\dots, x_n)$ with an isolated singularity at the origin.
There, the analogue of the set $\{q \in \ZZ \, | \, h^{p,q}(X) \neq 0\}$ above will be the set of the {\em exponents} of $f(x_1,\dots, x_n)$,  
which is a set of rational numbers and is also one of the most important numerical invariants defined by the mixed Hodge structure 
associated to $f(x_1,\dots, x_n)$. Let us give two important examples. 
First, suppose that $f(x_1,\dots, x_n)$ is a non-degenerate weighted homogeneous polynomial, namely, a polynomial with an isolated singularity at the origin with the property that 
there are positive rational numbers $w_i$, $i=1,\dots ,n$, such that
 $f(\lambda^{w_1}x_1,\dots ,\lambda^{w_n}x_n)=\lambda f(x_1,\dots,x_n)$, $\lambda\in\CC\backslash\{0\}$. 
We have the following properties of the exponents of $f$:
\begin{theorem}[cf.\ \cite{st:1}]\label{thm:st}
Let $q_1\le q_2\le \dots\le q_\mu$ be the exponents of $f$, where $\mu$ is the Milnor number of $f$ defined by
\[
\mu:=\dim_\CC\CC[x_1,\dots ,x_n]\left/\left(\frac{\partial f}{\partial x_1},\dots, \frac{\partial f}{\partial x_n}\right)\right. .
\]
Then one has 
\[
\mu= (-1)^n \prod_{i=1}^n\left(1-\frac{1}{w_i}\right)
\]
and
\[
\sum_{i=1}^\mu y^{q_i-\frac{n}{2}}=(-1)^n\prod_{i=1}^n\frac{y^{\frac{1}{2}}-y^{w_i-\frac{1}{2}}}{1-y^{w_i}}.
\]
In particular, one has a duality of exponents $q_i+q_{\mu-i+1}=n$, $i=1,\dots, \mu$, and hence
\[
\sum_{i=1}^\mu q_i=\frac{1}{2}n\cdot \mu.
\]
\end{theorem}

The following formula was proven by C.~Hertling \cite{H} in the context of Frobenius manifolds and an elementary proof was given by A.~Dimca \cite{D}.

\begin{theorem}[Hertling, Dimca]
Let $q_1\le q_2\le \dots\le q_\mu$ be the exponents of $f$. One has 
\[
\sum_{i=1}^\mu \left(q_i-\frac{n}{2}\right)^2=\frac{1}{12}\hat{c}\cdot \mu, \quad \hat{c}:=n-2\sum_{i=1}^n w_i.
\]
\end{theorem}
Next, consider the polynomial $f(x_1,x_2, x_3):=x_1^{\alpha_1}+x_2^{\alpha_2}+x_3^{\alpha_3}-x_1x_2x_3$ such that 
$1/\alpha_1+1/\alpha_2+1/\alpha_3<1$.
We have the following properties of the exponents of $f$:
\begin{theorem}[cf. \cite{AGV85}]\label{thm:cuspintro}
The set of exponents $\{q_i\}$ of $f$ is given by 
\begin{multline*}
\left\{1, \frac{1}{\alpha_1}+1,\frac{2}{\alpha_1}+1,\dots,\frac{\alpha_1-1}{\alpha_1}+1,
\frac{1}{\alpha_2}+1,\frac{2}{\alpha_2}+1,\dots,\right.\\
\left.\dots, \frac{\alpha_2-1}{\alpha_2}+1,
\frac{1}{\alpha_3}+1,\frac{2}{\alpha_3}+1,\dots,\frac{\alpha_3-1}{\alpha_3}+1,
2\right\}.
\end{multline*}
In particular, one has
\[
\sum_{i=1}^\mu \left(q_i-\frac{3}{2}\right)^2=\frac{1}{12}\mu+\frac{1}{6}\chi,
\quad \chi:=2+\sum_{i=1}^3\left(\frac{1}{\alpha_i}-1\right).
\]
\end{theorem}
The purpose of this paper is to generalize these results  to pairs 
$(f,G)$, where $G \subset {\rm SL}_n(\CC)$ is a finite abelian subgroup leaving $f$ invariant. If $f$ is weighted homogeneous, such a pair is also called an {\em orbifold Landau-Ginzburg model} because $f$ is the potential of such a model. Our main theorem in this paper is Theorem~\ref{thm:main}.
The generalization of  Theorem~\ref{thm:cuspintro} is given as Theorem~\ref{thm:cusp}.
The similarity between smooth compact K\"ahler manifolds and isolated hypersurface singularities with 
a group action is not an accident but a matter of course.
Mirror symmetry predicts a correspondence between Landau-Ginzburg models and (non-commutative) Calabi-Yau orbifolds.
For example, a mirror partner of a weighted homogeneous polynomial with a group action is 
a fractional Calabi--Yau manifold of dimension $\hat{c}$, which has lead us to the statement of  Theorem~\ref{thm:main}.
\smallskip

\begin{sloppypar}

{\bf Acknowledgements}.\  
This work has been supported 
by the DFG-programme SPP1388 ''Representation Theory'' (Eb 102/6-1).
The second named author is also supported 
by JSPS KAKENHI Grant Number 24684005. 
We thank the anonymous referee for carefully reading our paper and for most valuable comments.
\end{sloppypar}
\section{Basic properties of E-functions}
Let $G$ be a finite abelian subgroup of ${\rm SL}_n(\CC)$ acting diagonally on $\CC^n$.
For $g \in G$, we denote by ${\rm Fix}\, g :=\{x\in\CC^n~|~g\cdot x=x \}$ the fixed locus of $g$  
and by $n_g: = \dim {\rm Fix}\, g$ its dimension.

We first introduce the notion of the age of an element of a finite group as follows: 
\begin{definition}[\cite{IR}]
Let $g \in G$ be an element and $r$ be the order of $g$. Then $g$ has a unique expression of the following form  
\[ g={\rm diag}({\bf e}[a_1/r], \ldots, {\bf e}[a_n/r]) \quad \mbox{with } 0 \leq a_i < r, \]
where ${\bf e}[ - ] = e^{2 \pi \sqrt{-1} \cdot  -}$.
Such an element $g$ is often simply denoted by $g=\frac{1}{r}(a_1, \ldots, a_n)$. The {\em age} of $g$ is defined as 
\[ {\rm age}(g) := \frac{1}{r}\sum_{i=1}^n a_i. \] 
Since we assume that $G \subset {\rm SL}_n(\CC)$, the ${\rm age}(g)$ is a non-negative integer for all $g\in G$. 
\end{definition}
\begin{definition}
An element $g\in G$ of age 1 with ${\rm Fix}\, g =\{0\}$ is called a {\em junior element}. 
The number of junior elements is denoted by $j_G$.
\end{definition}
Let $f=f(x_1,\dots, x_n)$ be a polynomial with an isolated singularity at the origin 
which is invariant under the natural action of $G$. 
For $g \in G$, set $f^g:=f|_{{\rm Fix}\, g}$. 
\begin{proposition}
The function $f^g$ has an isolated singularity at the origin.
\end{proposition}
\begin{proof}
Since $G$ acts diagonally on $\CC^n$, we may assume that ${\rm Fix}\, g=\{x_{n_g+1}=\dots =x_{n}=0\}$ 
by a suitable renumbering of indices.
Since $f$ is invariant under $G$, $g\cdot x_i\ne x_i$ for $i=n_g+1,\dots, n$ and 
$\frac{\partial f}{\partial x_{n_g+1}},\dots, \frac{\partial f}{\partial x_n}$ form a regular sequence, we have
\[
\left(\frac{\partial f}{\partial x_{n_g+1}},\dots, \frac{\partial f}{\partial x_n}\right)\subset 
\left(x_{n_g+1},\dots, x_n\right).
\]
Therefore, we have
\begin{align*}
&\dim_\CC\CC[x_1,\dots ,x_{n_g}]\left/\left(\frac{\partial f^g}{\partial x_1},\dots, \frac{\partial f^g}{\partial x_{n_g}}\right)\right.\\
= &\dim_\CC\CC[x_1,\dots ,x_n]\left/\left(\frac{\partial f}{\partial x_1},\dots, \frac{\partial f}{\partial x_{n_g}},
x_{n_g+1},\dots, x_n\right)\right.\\
\le &\dim_\CC\CC[x_1,\dots ,x_n]\left/\left(\frac{\partial f}{\partial x_1},\dots, \frac{\partial f}{\partial x_n}\right)\right.<\infty.
\end{align*}
\end{proof}
We first associate to $f$ a natural mixed Hodge structure with an automorphism,
which gives the following bi-graded vector space: 
\begin{definition}
Define the bi-graded vector space $\H_f:=\displaystyle\bigoplus_{p,q\in\QQ}\H^{p,q}_f$ as
\begin{enumerate}
\item If $p+q\ne n$, then  $\H^{p,q}_f:=0$.
\item If $p+q=n$ and $p\in\ZZ$, then  
\[
\H^{p,q}_f:={\rm Gr}^{p}_{F^\bullet}H^{n-1}(Y_f,\CC)_1.
\]
\item If $p+q=n$ and $p\notin\ZZ$, then  
\[
\H^{p,q}_f:={\rm Gr}^{[p]}_{F^\bullet}H^{n-1}(Y_f,\CC)_{e^{2\pi\sqrt{-1} p}},
\]
where $[p]$ is the largest integer less than $p$.
\end{enumerate}
\end{definition}
As a vector space, $\H_f$ is identified with 
$\Omega_f:=\Omega_{\CC^n,0}^{n}\left/df\wedge \Omega_{\CC^n,0}^{n-1}\right. $.
Note that we have
\[
\Omega_f
= \O_{\CC^n,0}\left/\left(\frac{\partial f}{\partial x_1},\dots, \frac{\partial f}{\partial x_n}\right)\right. 
\cdot dx_1\wedge \dots \wedge dx_n,
\]
and that Theorem~\ref{thm:st} can be shown based on the above equality by calculating the Poincar\'{e} polynomial 
of the right hand side. 
The $G$-action on $\H_f$ can also be identified with the one on $\Omega_f$.
We shall use the fact that $\H_{f^g}$ admits a natural $G$-action 
by restricting the $G$-action on $\CC^n$ to ${\rm Fix}\, g$ (which is well-defined since $G$ acts diagonally on $\CC^n$).
To the pair $(f,G)$ we can associate a natural mixed Hodge structure with an automorphism,
which gives the following bi-graded vector space:
\begin{definition}
Define the bi-graded $\CC$-vector space $\H_{f,G}$ as 
\begin{equation}
\H_{f,G}:=\bigoplus_{g\in G}(\H_{f^g})^G(-{\rm age}(g),-{\rm age}(g)),
\end{equation}
where $(\H_{f^g})^G$ denotes the $G$-invariant subspace of $\H_{f^g}$.
\end{definition}
Since the bi-graded vector space $\H_{f,G}$ is the analog of $\displaystyle\bigoplus_{p,q\in\ZZ}H^q(X,\Omega_X^p)$ for a smooth compact K\"{a}hler 
manifold $X$, we introduce the following notion:
\begin{definition}
The {\em Hodge numbers} for the pair $(f,G)$ are  
\[
h^{p,q}(f,G):= \dim_\CC  \H^{p,q}_{f,G},\quad p,q\in\QQ.
\]
\end{definition}
\begin{definition}
The rational number $q$ with $\H^{p,q}_{f,G}\ne 0$ is called an {\em exponent} of the pair $(f,G)$. 
The {\em set of exponents} of the pair $(f,G)$ is the multi-set of exponents 
\[
\left\{q*h^{p,q}(f,G)~|~p,q\in\QQ,\ h^{p,q}(f,G)\ne 0 \right\},
\]
where by $u*v$ we denote $v$ copies of the rational number $u$.  
\end{definition}
Note that $p+q\in\ZZ$ for the rational number $q$ with $h^{p,q}({f,G})\ne 0$ since $G \subset {\rm SL}_n(\CC)$.
\begin{definition} The E-function for the pair $(f,G)$ is
\begin{equation}
E(f,G)(t,\bar{t}):=\sum_{p,q\in\QQ}(-1)^{(p-n)+q}h^{p,q}(f,G) \cdot 
t^{p-\frac{n}{2}}\bar{t}^{q-\frac{n}{2}}. 
\end{equation}
\end{definition}
\begin{definition}
The {\em Milnor number} for the pair $(f,G)$ is  
\[
\mu_{(f,G)}:=E(f,G)(1,1)= \sum_{p,q\in\QQ}(-1)^{(p-n)+q}h^{p,q}(f,G).
\]
\end{definition}
\begin{theorem} \label{thm:EGqh}
Assume that $f$ is a non-degenerate weighted homogeneous polynomial. 
Write $g \in G$ in the form 
$(\lambda_1(g), \ldots, \lambda_n(g))$ where $\lambda_i(g)={\bf e}[a_iw_i]$. The E-function for the pair $(f,G)$ is given by the following formula$:$
\begin{equation}
E(f,G)(t,\bar{t})=\sum_{g\in G}E_g(f,G)
(t,\bar{t}),
\end{equation}
\begin{equation}
E_g(f,G)(t,\bar{t})
:= (-1)^n \left(\prod_{a_iw_i \not\in \ZZ}\left({t}{\bar{t}}\right)^{w_ia_i-[w_ia_i]-\frac{1}{2}} \right)
\cdot \frac{1}{|G|} \sum_{h\in G}\prod_{a_iw_i\in\ZZ}
\frac{\left(\frac{\bar{t}}{t}\right)^{\frac{1}{2}}-\lambda_i(h)\left(\frac{\bar{t}}{t}\right)^{w_i-\frac{1}{2}}}
{1-\lambda_i(h)\left(\frac{\bar{t}}{t}\right)^{w_i}}.
\end{equation}
Here $[a]$ for $a \in \QQ$ denotes the largest integer less than or equal to $a$.
\end{theorem}
\begin{proof}
Theorem~\ref{thm:st} enables us to obtain $E_g(f,G)(t,\bar{t})$.
In particular, the term  
\[
\frac{1}{|G|} \sum_{h\in G}(-1)^{n_g}\prod_{a_iw_i\in\ZZ}
\frac{\left(\frac{\bar{t}}{t}\right)^{\frac{1}{2}}-\lambda_i(h)\left(\frac{\bar{t}}{t}\right)^{w_i-\frac{1}{2}}}
{1-\lambda_i(h)\left(\frac{\bar{t}}{t}\right)^{w_i}}
\]
calculates the $G$-invariant part of $E(f^g,\{1\})(t,\bar{t})$ and the term $\displaystyle\prod_{w_ia_i \not\in \ZZ} (-1) \left({t}{\bar{t}}\right)^{w_ia_i-[w_ia_i]-\frac{1}{2}}$ gives the contribution from the age shift $(-{\rm age}(g),-{\rm age}(g))$.
\end{proof}
We have the following properties of the Hodge numbers $h^{p,q}(f,G)$.
\begin{corollary}
Assume that $f$ is a non-degenerate weighted homogeneous polynomial. 
We have 
\[
h^{p,q}(f,G)=h^{q,p}(f,G),\quad p,q\in\QQ.
\]
In other words, we have 
\[
E(f,G)(t,\bar{t})=E(f,G)(\bar{t},t).
\]
\end{corollary}
\begin{proof}
This is shown by an elementary direct calculation.
\end{proof}
\begin{corollary}
Assume that $f$ is a non-degenerate weighted homogeneous polynomial. 
The Hodge numbers satisfy the ``Serre duality" 
\[
h^{p,q}(f,G)=h^{n-p,n-q}(f,G),\quad p,q\in\QQ.
\]
In other words, we have 
\[
E(f,G)(t,\bar{t})=E(f,G)(t^{-1},\bar{t}^{-1}).
\]
\end{corollary}
\begin{proof}
By using the formula 
\[
w_i(-a_i)-[w_i(-a_i)]-\frac{1}{2}=-w_i a_i+[w_i a_i]+\frac{1}{2},
\]
an easy calculation yields the formula.
\end{proof}
\begin{corollary}
Assume that $f$ is a non-degenerate weighted homogeneous polynomial. 
The mean of the set of exponents of $(f,G)$ is $n/2$.
Namely, we have 
\[
\sum_{p,q\in\QQ}(-1)^{(p-n)+q} \left(q-\frac{n}{2} \right)h^{p,q}(f,G)=0.
\]
\end{corollary}
\begin{proof}
This is obvious from the previous corollary.
\end{proof}
\begin{definition}
Define the {\em variance of the set of exponents} of $(f,G)$ by 
\[ {\rm Var}_{(f,G)} := \sum_{p,q\in\QQ}(-1)^{(p-n)+q} \left(q- \frac{n}{2} \right)^2 h^{p,q}(f,G). \]
\end{definition}
In order to state our formula for the variance, we introduce the following notion of dimension for a polynomial $f$ with 
an isolated singularity at the origin.
\begin{definition}
The non-negative rational number $\hat{c}$ defined as the difference of the maximal exponent of the pair $(f,\{1\})$ 
and the minimal exponent of the pair $(f,\{1\})$ is called the {\em dimension} of $f$.
\end{definition}
\begin{proposition}
Assume that $f$ is a non-degenerate weighted homogeneous polynomial. 
The dimension $\hat{c}$ of $f$ is given by 
\[
\hat{c} := n - 2\sum_{i=1}^n w_i.
\]
\end{proposition}
\begin{proof}
It easily follows from Theorem~\ref{thm:st} that the maximal exponent and the minimal exponent are 
given by $n-\displaystyle\sum_{i=1}^n w_i$ and $\displaystyle\sum_{i=1}^n w_i$ respectively.
\end{proof}
It is natural from the mirror symmetry point of view to expect that
the variance of the set of exponents of $(f,G)$ should be given by 
\begin{equation}
{\rm Var}_{(f,G)} = \frac{1}{12} \hat{c} \cdot\mu_{(f,G)}.
\end{equation}
This will be proved in the next section.
\section{Variance of the exponents}
\begin{definition}
The {\em $\chi_y$-genus} for the pair $(f,G)$ is  
\[ \chi(f,G)(y):=E(f,G)(1,y).
\]
We have 
\[ \chi(f,G)(y) = (-1)^n \sum_{g \in G}\left( y^{{\rm age}(g) - \frac{n-n_g}{2}}
\cdot \frac{1}{|G|} \sum_{h\in G}\prod_{\lambda_i(g) = 1}
\frac{y^{\frac{1}{2}}-\lambda_i(h)y^{w_i-\frac{1}{2}}}
{1-\lambda_i(h)y^{w_i}} \right) .
\]
\end{definition}
One has 
\begin{eqnarray*}
\mu_{(f,G)} & = &\lim_{y \to 1} \chi(f,G)(y), \\
{\rm Var}_{(f,G)} & = & \lim_{y \to 1} \frac{d}{dy}\left(y\frac{d}{dy}\chi(f,G)(y)\right).
\end{eqnarray*}

\begin{proposition} \label{prop:lim}
Let 
\[ p_i(y) := \frac{y^{\frac{1}{2}}-\lambda_i(h)y^{w_i-\frac{1}{2}}}{1-\lambda_i(h)y^{w_i}}. \]
{\rm (i)} For $\lambda_i(h)=1$ one has
\[ \lim_{y \to 1} p_i(y)
= 1 - \frac{1}{w_i}, \quad  \lim_{y \to 1}  \frac{\frac{d}{dy} p_i(y)}{p_i(y)} = 0, \quad
\lim_{y \to 1} \frac{d}{dy}\left( y \frac{\frac{d}{dy} p_i(y)}{p_i(y)} \right) = \frac{1-2w_i}{12}. 
\]
{\rm (ii)} For $\lambda_i(h) \neq 1$ one has
\[ \lim_{y \to 1} p_i(y) = 1, 
\quad  \lim_{y \to 1}  \frac{\frac{d}{dy} p_i(y)}{p_i(y)} = \frac{1}{2} \frac{1+ \lambda_i(h)}{1-\lambda_i(h)}, \quad  \lim_{y \to 1} \frac{d}{dy}\left( y \frac{\frac{d}{dy} p_i(y)}{p_i(y)} \right) = - \frac{(1-2w_i) \lambda_i(h)}{(1-\lambda_i(h))^2}. 
\]
\end{proposition}

\begin{proof} For (i) see the proof of \cite[Proposition~5.2]{D}. Statement (ii) follows from a similar elementary but tedious computation. 
\end{proof}

Let $I_0:=\{1, \ldots, n \}$ and let $H \subset G$ be a subgroup of $G$. For a subset $I \subset I_0$ ($I = \emptyset$ is admitted)  let $H^I$ be the maximal subgroup of $H$ fixing the coordinates $x_i$, $i \in I$.

\begin{lemma} \label{lem:p'}
Let $H \subset G$ be a subgroup of $G$ and $i \in I_0$. Then
\[ \sum_{h \in H \setminus H^{\{ i \}}} \frac{1+ \lambda_i(h)}{1-\lambda_i(h)} =0
\]
\end{lemma}

\begin{proof} One has
\[ \sum_{h \in H \setminus H^{\{ i \}}} \frac{1+ \lambda_i(h)}{1-\lambda_i(h)} = \sum_{h \in H \setminus H^{\{ i \}}} \frac{1}{1-\lambda_i(h)} + \sum_{h \in H \setminus H^{\{ i \}}} \frac{1}{\lambda_i(h^{-1})-1} = 0.
\]
\end{proof}

\begin{proposition} \label{prop:zeta}
Let $r \in \ZZ$, $r \geq 2$, and $\zeta_r= {\bf e}[1/r]$ be a primitive $r$-th root of unity. Then one has
\[ -\sum_{k=1}^{r-1}  \frac{\zeta_r^k}{(1-\zeta_r^k)^2} = \frac{r^2-1}{12}. \]
\end{proposition}

\begin{proof} One has
\[ -\sum_{k=1}^{r-1}  \frac{\zeta_r^k}{(1-\zeta_r^k)^2} = \lim_{t \to 1} q'(t) \mbox{ where } q(t):= - \sum_{k=1}^{r-1} \frac{1}{1- \zeta_r^k t}. \]
One can easily see that
\[ q(t) = \frac{- r \left( \sum_{k=0}^{r-2} t^k \right) + \sum_{k=0}^{r-2} (k+1) t^k}{\sum_{k=0}^{r-1} t^k}. \]
This implies
\[ \lim_{t \to 1} q'(t) = \frac{1}{r^2} \left[  \sum_{k =1}^{r-2} k(k-r+1)  r - \left( \sum_{\ell=1}^{r-1} (\ell-r) \right) \left( \sum_{k=1}^{r-1} k \right) \right] = \frac{r^2-1}{12}. \]
\end{proof}

\begin{corollary} \label{cor:sum}
Let $H \subset G$ be a subgroup of $G$ and $i \in I_0$. Then
\[ - \sum_{h \in H \setminus H^{\{ i \}}} \frac{\lambda_i(h)}{(1-\lambda_i(h))^2} = \frac{|H \cap H^{\{ i \}}|(|H/H \cap H^{\{ i \}}|^2-1)}{12}
\]
\end{corollary}

\begin{proof}
The image of the factor group $H/H \cap H^{\{ i \}}$ under the induced character $\lambda_i : H/H \cap H^{\{ i \}} \to \CC^\ast$ is a finite abelian subgroup of the unit circle $S^1$ and hence cyclic. Therefore the formula follows from Proposition~\ref{prop:zeta}.
\end{proof}

Let 
\[ ((x)) := \left\{ \begin{array}{cl} x - [x] - \frac{1}{2} & \mbox{if } x \in \RR, x \not\in \ZZ, \\
0 & \mbox{if } x \in \ZZ. \end{array} \right. 
\]

\begin{proposition} \label{prop:Dedekind}
Let $r \in \ZZ$, $r \geq 2$, $\zeta_r= {\bf e}[1/r]$ be a primitive $r$-th root of unity, and $a,b$ be integers satisfying $0 < a,b < r$. Then one has
\[
\frac{1}{4r} \sum _{k=1, \atop r \not \, | \, ak, bk}^{r-1} \frac{1+\zeta_r^{ak}}{1-\zeta_r^{ak}} \frac{1+\zeta_r^{bk}}{1-\zeta_r^{bk}} = - \sum_{k=1}^{r-1} ((\frac{ak}{r}))((\frac{bk}{r})). 
\]
\end{proposition}

\begin{remark}
The right hand side of the formula of Proposition~\ref{prop:Dedekind} is a generalized Dedekind sum and Proposition~\ref{prop:Dedekind} is a slight generalization of \cite[5.2 Theorem 1]{HZ}, since
\[  \frac{1+ {\bf e}[x]}{1-{\bf e}[x]} = \sqrt{-1} \cot  \pi x
\]
for any real number $x$.
The difference is that \cite[5.2 Theorem 1]{HZ} is only formulated for integers $a,b$ prime to $r$.
\end{remark}

\begin{proof}[Proof of Proposition~\ref{prop:Dedekind}]
We follow the proof of \cite[5.2 Theorem 1]{HZ}. For simplicity, we assume $b=1$. By the formula \cite[5.2 (2)]{HZ} which goes back to Eisenstein \cite{Ei}, we have
\[ (( \frac{q}{r} )) = - \frac{1}{2r} \sum_{\ell=1}^{r-1} \zeta_r^{\ell q} \frac{\zeta_r^\ell+1}{\zeta_r^\ell-1}\]
for any integers $q$ and $r$.
(Note that there is a minor misprint in \cite[5.2 (2)]{HZ}.) Applying this formula, we get
\begin{eqnarray*}
\sum_{\ell=1}^{r-1} ((\frac{a\ell}{r}))((\frac{\ell}{r})) & = & \sum_{\ell=1}^r ((\frac{a\ell}{r}))((\frac{\ell}{r})) =\frac{1}{4r^2}   \sum_{\ell=1}^r \sum_{m=1}^{r-1}  \sum_{k=1}^{r-1}  \zeta_r^{(m+ak)\ell} \frac{\zeta_r^m+1}{\zeta_r^m-1} \frac{\zeta_r^k+1}{\zeta_r^k-1} \\
& = & \frac{1}{4r} \sum_{k=1, \atop r \not \, | \, ak}^{r-1} \frac{\zeta_r^{-ak}+1}{\zeta_r^{-ak}-1} \frac{\zeta_r^k+1}{\zeta_r^k-1} 
 =  - \frac{1}{4r} \sum_{k=1, \atop r \not \, | \, ak}^{r-1} \frac{1+\zeta_r^{ak}}{1-\zeta_r^{ak}} \frac{1+\zeta_r^k}{1-\zeta_r^k},
\end{eqnarray*}
since
\[   \sum_{\ell=1}^r \zeta_r^{(m+ak)\ell} = \left\{ \begin{array}{cl} 0 & \mbox{if } m+ak \not\equiv 0 \, \mbox{mod} \, r ,\\
r &  \mbox{if } m+ak \equiv 0 \, \mbox{mod} \, r . \end{array} \right.
\]
\end{proof}

\begin{corollary} \label{cor:Dedekind} Let $K \subset J \subset I_0$. Then 
\[  \frac{1}{4} \sum_{h \in G^K} \left( \sum_{j \in J \setminus K, \atop \lambda_j(h) \neq 1} \frac{1+\lambda_j(h)}{1-\lambda_j(h)} \right)^2 = - |G^K| \sum_{h \in G^K}  \left( \sum_{j \in J \setminus K} ((a_jw_j)) \right)^2, \]
where $\lambda_j(h) = {\bf e}[a_jw_j]$ for all $h \in G^K$ and $j \in J \setminus K$.
\end{corollary}

\begin{proof} This follows from Proposition~\ref{prop:Dedekind} by the same arguments as in the proof of Corollary~\ref{cor:sum}.
\end{proof}

\begin{proposition} \label{prop:mu} 
One has
\begin{equation}
 \mu_{(f,G)} = \frac{(-1)^n}{|G|} \left\{ \sum_{I \subset I_0} \prod_{i \in I} \left( 1 - \frac{1}{w_i} \right) \left[ \sum_{I \subset J \subset I_0} (-1)^{|J|-|I|} \left| G^J \right|^2 \right] \right\} .
\end{equation}
\end{proposition}

\begin{proof} Let $J \subset I_0$. Let $G_J$ be the set of elements of $g \in G$ with $\lambda_j(g)=1$ for $j \in J$ and $\lambda_j(g) \neq 1$ for $j \not\in J$, i.e. the set of elements of $G$ which fix the coordinates $x_j$, $j \in J$, and only these coordinates. Then
\[ |G_J| = \sum_{K, \atop J \subset K \subset I_0} (-1)^{|K|-|J|} |G^K|. \]
Let $I \subset J$. Let $G_{I,J}$ be the set of elements $g$ of $G$ with $\lambda_i(g)=1$ for $i \in I$ and $\Lambda_j(g) \neq 1$ for $j \in J \setminus I$ (and $\lambda_k(g)$ arbitrary for $k \in I_0 \setminus J$). Then
\[ |G_{I,J}| = \sum_{K, \atop I \subset K \subset J} (-1)^{|K|-|I|} |G^K|. \]
By Proposition~\ref{prop:lim} one has
\begin{eqnarray*}
\lim_{y \to 1} \chi(f,G)(y)  & =  & \frac{(-1)^n}{|G|} \sum_{J, \atop J \subset I_0} |G_J| \left( \sum_{I, \atop I \subset J} \prod_{i \in I} \left( 1- \frac{1}{w_i} \right)  |G_{I,J}|  \right) \\
& =  & {}  \frac{(-1)^n}{|G|} \sum_{I, \atop I \subset I_0} \prod_{i \in I} \left( 1- \frac{1}{w_i} \right) \left( \sum_{J, \atop I \subset J \subset I_0} |G_J||G_{I,J}| \right).
\end{eqnarray*}
Now let $I \subset I_0$ be fixed. Then
\begin{eqnarray*}
\sum_{J, \atop I \subset J \subset I_0} |G_J||G_{I,J}|
& = & \sum_{J, \atop I \subset J \subset I_0} \left( \sum_{K, \atop J \subset K \subset I_0} (-1)^{|K|-|J|} |G^K| \right) \left( \sum_{L, \atop I \subset L \subset J} (-1)^{|L|-|I|} |G^L| \right) \\
& = & {} \sum_{L, \atop I \subset L \subset I_0} \sum_{K, \atop L \subset K \subset I_0} \left( \sum_{J, \atop L \subset J \subset K} (-1)^{|K|+|L|-|I|-|J|} \right) |G^K||G^L| \\
& = & {} \sum_{K, \atop I \subset K \subset I_0} (-1)^{|K|-|I|} |G^K|^2,
\end{eqnarray*}
since for fixed $L \subset I_0$ and $K \subset I_0$ with $L \subset K$
\begin{equation} \label{eq:sign}
\sum_{J, \atop L \subset J \subset K} (-1)^{|K|+|L|-|I|-|J|} = (-1)^{|K|-|I|} (1-1)^{|K|-|L|}= \left\{ \begin{array}{cl} (-1)^{|K|-|I|} & \mbox{for } L=K, \\
0 & \mbox{otherwise.} \end{array} \right.
\end{equation}
\end{proof}
Now we are ready to state the main result of our paper.
\begin{theorem}\label{thm:main}
One has
\[ {\rm Var}_{(f,G)} =\sum_{p,q\in\QQ}(-1)^{(p-n)+q} \left(q- \frac{n}{2} \right)^2 h^{p,q}(f,G)
= \frac{1}{12} \hat{c} \cdot\mu_{(f,G)}. \]
\end{theorem}

\begin{proof} We use the notation introduced in the proof of Proposition~\ref{prop:mu}. 
By Proposition~\ref{prop:lim} and Lemma~\ref{lem:p'} we have
\[ \lim_{y \to 1}\frac{d}{dy} \left(y\frac{d}{dy}\chi(f,G)(y)\right) = A + B + C, \]
where
\begin{eqnarray*} 
A & := & \frac{(-1)^n}{|G|}  \sum_{J, \atop J \subset I_0} \sum_{g \in G_J} \left( {\rm age}(g) - \frac{n-n_g}{2}\right)^2 \left[ \sum_{I, \atop I \subset J} \prod_{i \in I} \left( 1- \frac{1}{w_i} \right) |G_{I,J}| \right], \\
B & := & \frac{(-1)^n}{|G|} \sum_{J, \atop J \subset I_0} |G_J| \left[ \sum_{I, \atop I \subset J} \prod_{i \in I} \left( 1- \frac{1}{w_i} \right) \sum_{h \in G_{I,J}} \frac{1}{4} \left( \sum_{j \in J \setminus I} \frac{1+ \lambda_j(h)}{1- \lambda_j(h)} \right)^2 \right], \\
C & := & \frac{(-1)^n}{|G|} \sum_{J, \atop J \subset I_0} |G_J|  \times \\
 &  & {}
\left[ \sum_{I, \atop I \subset J} \prod_{i \in I} \left( 1- \frac{1}{w_i} \right) \left( |G_{I,J}| \left( \sum_{i \in I} \frac{1-2w_i}{12} \right) - \sum_{h \in G_{I,J}} \sum_{j \in J, \atop j \not\in I}   \frac{(1-2w_j)\lambda_j(h)}{(1-\lambda_j(h))^2} \right) \right].
\end{eqnarray*}

a) We first show that $A+B=0$. We first take the sums in $A$ and $B$ in a different order:
\begin{eqnarray*}
A & = & \frac{(-1)^n}{|G|} \sum_{I, \atop I \subset I_0} \prod_{i \in I} \left( 1- \frac{1}{w_i} \right) A_I, \ A_I:= \sum_{J, \atop I \subset J \subset I_0} \sum_{g \in G_J} \left( {\rm age}(g) - \frac{n-n_g}{2}\right)^2 |G_{I,J}|, \\
B & = & \frac{(-1)^n}{|G|} \sum_{I, \atop I \subset I_0} \prod_{i \in I} \left( 1- \frac{1}{w_i} \right) B_I, \ B_I :=  \sum_{J, \atop I \subset J \subset I_0} |G_J| \left( \sum_{h \in G_{I,J}} \frac{1}{4} \left( \sum_{j \in J \setminus I} \frac{1+ \lambda_j(h)}{1- \lambda_j(h)} \right)^2 \right).
\end{eqnarray*}
Now let $I \subset I_0$ be fixed. Let $\lambda_i(g)={\bf e}[a_iw_i]$. Then we have on one hand:
\begin{eqnarray*}
A_I & = &  \sum_{J, \atop I \subset J \subset I_0} |G_{I,J}| \sum_{g \in G_J} \left( \sum_{j \in I_0 \setminus J} ((a_jw_j)) \right)^2 \\
& = & \sum_{J, \atop I \subset J \subset I_0} |G_{I,J}| \sum_{K, \atop J \subset K \subset I_0} (-1)^{|K|-|J|}  \sum_{g \in G^K} \left( \sum_{j \in I_0 \setminus K} ((a_jw_j)) \right)^2.
\end{eqnarray*}
On the other hand we have by Corollary~\ref{cor:Dedekind}
\begin{eqnarray*}
B_I & = & \sum_{J, \atop I \subset J \subset I_0} |G_J|  \sum_{h \in G_{I,J}} \frac{1}{4} \left( \sum_{j \in J \setminus I} \frac{1+ \lambda_j(h)}{1- \lambda_j(h)} \right)^2  \\
& = & \sum_{J, \atop I \subset J \subset I_0} |G_J| \sum_{K, \atop I \subset K \subset J} (-1)^{|K|-|I|}   \sum_{h \in G^K}  \frac{1}{4} \left( \sum_{j \in J \setminus K} \frac{1+ \lambda_j(h)}{1- \lambda_j(h)} \right)^2 \\
& = & - \sum_{J, \atop I \subset J \subset I_0} |G_J| \sum_{K, \atop I \subset K \subset J} (-1)^{|K|-|I|} |G^K| \sum_{h \in G^K} \left( \sum_{j \in J \setminus K} ((a_jw_j)) \right)^2 .
\end{eqnarray*}
For $I \subset K \subset J  \subset I_0$ let 
\[ s(K,J):=\sum_{g \in G^K} \left( \sum_{j \in J \setminus K} ((a_jw_j)) \right)^2 \]
Then 
\begin{eqnarray*}
A_I & = & \sum_{K, \atop I \subset K \subset I_0} \sum_{J, \atop  I \subset J \subset K} (-1)^{|K|-|J|}  |G_{I,J}| s(K,I_0) \\
 & := & \sum_{K, \atop I \subset K \subset I_0} \sum_{J, \atop  I \subset J \subset K} (-1)^{|K|-|J|} \left( \sum_{L, \atop I \subset L \subset J} (-1)^{|L|-|I|} |G^L| \right) s(K,I_0)\\
& = & \sum_{L, \atop I \subset L \subset I_0} \sum_{K, \atop L \subset K \subset I_0} \left( \sum_{J, \atop  L \subset J \subset K} (-1)^{|K|+|L|-|I|-|J|} \right) |G^L| s(K,I_0)\\
& = & \sum_{K, \atop I \subset K \subset I_0} (-1)^{|K|-|I|} |G^K| s(K,I_0)
\end{eqnarray*}
by Formula~(\ref{eq:sign}). 
On the other hand, we have
\begin{eqnarray*}
B_I & = & - \sum_{K, \atop I \subset K \subset I_0} \sum_{J, \atop K \subset J \subset I_0} (-1)^{|K|-|I|} |G_J| |G^K| s(K,J) \\
& = & - \sum_{K, \atop I \subset K \subset I_0} \sum_{J, \atop K \subset J \subset I_0} (-1)^{|K|-|I|} \left( \sum_{L, \atop J \subset L \subset I_0} (-1)^{|L|-|J|} |G^L| \right) |G^K| s(K,J) \\
& = & - \sum_{L, \atop I \subset L \subset I_0}\sum_{K, \atop I \subset K \subset L}  \left(  \sum_{J, \atop L \subset J \subset I_0} (-1)^{|K|+|L|-|I|-|J|} \right) |G^L| |G^K| s(K,J) \\
 & = & - \sum_{K, \atop I \subset K \subset I_0} (-1)^{|K|-|I|} |G^K| s(K,I_0) = - A_I,
\end{eqnarray*}
again by Formula~(\ref{eq:sign}) and since $|G^{I_0}|=1$. This shows that $A+B=0$.

b) We now consider the term $C$.
Let $J \subset I_0$, $I \subset J$ and $j \in J$, $j \not\in I$. Then it follows from Corollary~\ref{cor:sum} that
\[ - \sum_{h \in G_{I,J}} \frac{\lambda_j(h)}{(1-\lambda_j(h))^2} = \frac{1}{12} m_{I,j}^J, \]
where
\[ m_{I,j}^J := 
\sum_{K, j \not\in K, \atop I \subset K \subset J} (-1)^{|K|-|I|}|G^{K \cup \{ i\} }| \left( \left| G^K/G^{K \cup \{ i\} }\right|^2 -1 \right). \]
By a) we have
\begin{eqnarray*}  
\lefteqn{\lim_{y \to 1}\frac{d}{dy} \left(y\frac{d}{dy}\chi(f,G)(y)\right) = C } \\
& = & \frac{(-1)^n}{|G|} \sum_{J, \atop J \subset I_0} |G_J| \left[ \sum_{I, \atop I \subset J} \prod_{i \in I} \left( 1- \frac{1}{w_i} \right) \left( |G_{I,J}| \left( \sum_{i \in I} \frac{1-2w_i}{12} \right) + \sum_{j \in J, \atop j \not\in I} m_{I,j}^J \left( \frac{1-2w_j}{12} \right) \right) \right] \\
& =  & {}  \frac{(-1)^n}{|G|} \sum_{I, \atop I \subset I_0} \prod_{i \in I} \left( 1- \frac{1}{w_i} \right) \left[ \sum_{J, \atop I \subset J \subset I_0} |G_J| \left( |G_{I,J}| \left( \sum_{i \in I} \frac{1-2w_i}{12} \right) + \sum_{j \in J, \atop j \not\in I} m_{I,j}^J \left( \frac{1-2w_j}{12} \right) \right) \right] .
\end{eqnarray*}
Now let $I \subset I_0$ and $j \not\in I$ be fixed. Then
\begin{eqnarray*}
\lefteqn{ \sum_{J, j \in J, \atop I \subset J \subset I_0}  |G_J] m_{I,j}^J }\\
& = & \sum_{J, j \in J, \atop I \subset J \subset I_0} \left( \sum_{K, \atop J \subset K \subset I_0} (-1)^{|K|-|J|}|G^K| \right) \left( \sum_{L, j \not\in L, \atop I \subset L \subset J} (-1)^{|L|-|I|} |G^{L \cup \{ j \}}| \left( \left| G^L/G^{L \cup \{ j \}} \right|^2 -1 \right) \right) \\
 & = & {} \sum_{L, j \not\in L, \atop I \subset L \subset I_0} \sum_{K, j \in K, \atop L \subset K \subset I_0} \left( \sum_{J, j \in J, \atop L \subset J \subset K} (-1)^{|K|+|L|-|I|-|J|} \right) |G^K||G^{L \cup \{ j \}}| \left( \left| G^L/G^{L \cup \{ j \}} \right|^2 -1 \right).
\end{eqnarray*}
Since $j \not\in L$ but $j \in J$, the case $J=L$ and hence also $K=L$ is excluded in the sum
\[  \sum_{J, j \in J, \atop L \subset J \subset K} (-1)^{|K|+|L|-|I|-|J|}. \]
Therefore
\[  \sum_{J, j \in J, \atop L \subset J \subset K} (-1)^{|K|+|L|-|I|-|J|} = \left\{ \begin{array}{cl} (-1)^{|L|-|I|} & \mbox{for } K=L \cup \{ j \}, \\
0 & \mbox{otherwise.} \end{array} \right.
\]
Hence we obtain
\begin{eqnarray*}
\sum_{J, j \in J, \atop I \subset J \subset I_0}  |G_J] m_{I,j}^J
& = & \sum_{L, j \not\in L, \atop I \subset L \subset I_0} (-1)^{|L|-|I|} |G^{L \cup \{ j \}}|^2 \left( \left| G^L/G^{L \cup \{ j \}} \right|^2 -1 \right). \\
& = & {} \sum_{L, j \not\in L, \atop I \subset L \subset I_0} (-1)^{|L|-|I|} \left( |G^L|^2 - |G^{L \cup \{ j \}}|^2 \right) \\
& = & {} \sum_{K, \atop I \subset K \subset I_0} (-1)^{|K|-|I|} |G^K|^2.
\end{eqnarray*}
Therefore the statement follows from Proposition~\ref{prop:mu}.
\end{proof}
\section{Variance of the exponents for cusp singularities with group actions}
Let $f(x_1,x_2,x_3):= x_1^{\alpha_1} + x_2^{\alpha_2} + x_3^{\alpha_3} - x_1x_2x_3$ and $G$ be a finite subgroup of $SL_n(\CC)$ acting diagonally on 
$\CC^n$ under which $f$ is invariant. Let $K_i \subset G$ be the maximal subgroup fixing the coordinate $x_i$, $i=1,2,3$. Define numbers $\gamma_1, \ldots , \gamma_s$ by
\[ (\gamma_1, \ldots, \gamma_s) = \left( \frac{\alpha_i}{|G/K_i|} \ast |K_i|, i=1,2,3 \right) ,\]
where we omit numbers which are equal to one on the right-hand side. 
Define a number $\chi_{(f,G)}$ by 
\[
\chi_{(f,G)}:=2-2j_G+\sum_{i=1}^s\left(\frac{1}{\gamma_i}-1\right).
\]
\begin{lemma} \label{lem:cusp}
Let the pair $(f,G)$ be as above.
\begin{enumerate}
\item The Milnor number of the pair $(f,G)$ is given by
\begin{equation}
\mu_{(f,G)} = 2- 2j_G + \sum_{i=1}^s (\gamma_i -1). 
\end{equation}
\item The set of exponents for the pair $(f,G)$ is given by 
\begin{multline}
\left\{1,2\right\}\coprod \left\{\frac{1}{\gamma_1}+1,\frac{2}{\gamma_1}+1,\dots, 
\frac{\gamma_1-1}{\gamma_1}+1\right\}\\
\coprod \left\{\frac{1}{\gamma_2}+1,\frac{2}{\gamma_2}+1,\dots, 
\frac{\gamma_2-1}{\gamma_2}+1\right\}\coprod\dots\\
\dots\coprod \left\{\frac{1}{\gamma_s}+1,\frac{2}{\gamma_s}+1,\dots, 
\frac{\gamma_s-1}{\gamma_s}+1\right\}
\end{multline}
\end{enumerate}
\end{lemma}
\begin{proof}
See Corollary 5.13 and the proof of Theorem 5.12 of \cite{ET}.
\end{proof}
We have the following formula for the variance.
Note that we have $\hat{c}=1$ by Theorem~\ref{thm:cuspintro}.
\begin{theorem} \label{thm:cusp}
Let the pair $(f,G)$ be as above.
The variance of the set of exponents of $(f,G)$ is given by 
\begin{equation}\label{eq:cusp}
{\rm Var}_{(f,G)}= \frac{1}{12} \mu_{(f,G)}+\frac{1}{6}\chi_{(f,G)}=
\frac{1}{12}\hat{c}\cdot\mu_{(f,G)}+\frac{1}{6}\chi_{(f,G)}.
\end{equation}
\end{theorem}
\begin{proof}
Some elementary calculation yields the statement.
\end{proof}
Note that the pair $(f,G)$ can be considered as a mirror partner of the orbifold curve (Deligne--Mumford stack) $\C$
which is a smooth projective curve of genus $j_G$ with $s$ isotropic points of orders $\gamma_1, \ldots, \gamma_s$ 
(cf. Theorem 7.1 of \cite{ET}).
The above formula for the variance is compatible with this observation. 
In particular, the dimension of $\C$ is $1$,  $\mu_{(f,G)}$ is the orbifold Euler number $\chi(\C)$ of $\C$ and 
$\chi_{(f,G)}$ is the orbifold Euler characteristic of $\C$, which is the degree of the first Chern class 
$c_1(\C)$ of $\C$.
Applying this to the formula in Theorem~\ref{thm:lwb}, we recover the equation \eqref{eq:cusp}.  


\begin{thebibliography}{AGV}

\linespread{0.75}
\setlength{\parskip}{0.0ex}
\small{


\bibitem[AGV]{AGV85} V.~I.~Arnold,   S.~M.~Gusein-Zade, and A.~N.~Varchenko:
\emph{Singularities of Differentiable Maps}, Volume II, Birkh\"auser,
Boston Basel Berlin 1988.



\bibitem[B]{B} L.~Borisov: On Betti numbers and Chern classes of varieties with trivial odd cohomology groups. arXiv: alg-geom/9703023.


\bibitem[D]{D} A.~Dimca: Monodromy and Hodge theory of regular functions. In: New developments in singularity theory (Cambridge, 2000), NATO Sci. Ser. II Math. Phys. Chem., 21, Kluwer Acad. Publ., Dordrecht, 2001, pp.~257--278.

\bibitem[ET]{ET} W.~Ebeling, A.~Takahashi: Mirror symmetry between orbifold curves and cusp singularities with group action. Int. Math. Res. Not. doi: 10.1093/imrn/rns115.

\bibitem[E]{Ei} G.~Eisenstein: Th\'eor\`emes arithm\'etiques. J. Reine Angew. Math. {\bf 27} (1844), 
281--283.

\bibitem[H]{H} C.~Hertling: Frobenius manifolds and variance of the spectral numbers. In: New developments in singularity theory (Cambridge, 2000), NATO Sci. Ser. II Math. Phys. Chem., 21, Kluwer Acad. Publ., Dordrecht, 2001, pp.~235--255.

\bibitem[HZ]{HZ} F.~Hirzebruch, D.~Zagier: \emph{The Atiyah-Singer theorem and elementary number theory}. Publish or Perish, Inc., Berkeley, 1974.

\bibitem[IR]{IR} Y.~Ito, M.~Reid: The McKay correspondence for finite subgroups of ${\rm SL}(3,\CC)$. In: Higher-dimensional complex varieties (Trento, 1994), de Gruyter, Berlin, 1996, pp.~221--240.

\bibitem[LW]{LW} A.~Libgober, J.~Wood: Uniqueness of the complex structure on K\"{a}hler manifolds of certain homotopy types. J. Differential Geom. {\bf 32} (1990), no. 1, 139--154.

\bibitem[St]{st:1}
	J.~H.~M.~Steenbrink: Mixed Hodge structure on the vanishing cohomology. In: Real and Complex Singularities, Proc. Nordic Summer school, Oslo, (1976), pp.~525--563.

	orbifoldized Poincar\'{e} polynomials.
	Commun. Math. Phys. {\bf 205} (1999), 571--586.



}
\end{thebibliography}
\end{document}